\documentclass[10pt,twocolumn,a4paper]{article}

\newcommand{\authorname}{\textbf}
\newcommand{\authorgroup}{\emph}
\newcommand{\authoraddress}{\emph}
\newcommand{\authormail}{\emph}
\usepackage[utf8]{inputenx}
\usepackage[compact]{titlesec}
\usepackage[english]{babel}
\usepackage[pdftex]{graphicx}
\graphicspath{{../img/}}
\DeclareGraphicsExtensions{.pdf,.jpeg,.png}

\usepackage{amsmath,amssymb,mathrsfs}
\usepackage{amsthm}
\usepackage{fancyhdr}
\usepackage{url}
\usepackage{calc}
\usepackage{titling}
\usepackage[T1]{fontenc}
\usepackage{mathptmx}
\usepackage[superscript]{cite}
\usepackage{titlesec}
\usepackage[none]{hyphenat}
\usepackage[hyperfootnotes=false]{hyperref}

\newtheorem{proposition}{Proposition}[section]

\titleformat*{\section}{\normalfont\normalsize\bfseries}
\titleformat*{\subsection}{\normalfont\normalsize\em}
\fancyheadoffset[LE,RO]{\marginparsep + \marginparwidth}
\pretitle{\begin{center}\Large}
\posttitle{\end{center}}

\preauthor{\begin{center}\normalsize \lineskip 0.5em\begin{tabular}[t]{c}}
\postauthor{\end{tabular}\par\end{center}}

\predate{\begin{center}\tiny}
\postdate{\par\end{center}}
\newcommand{\mathbbm}[1]{\text{\usefont{U}{bbm}{m}{n}#1}}

\newcommand\blfootnote[1]{%
  \begingroup
  \renewcommand\thefootnote{}\footnote{#1}%
  \addtocounter{footnote}{-1}%
  \endgroup
}

\title{\vspace{-5ex}\textbf{A Bilevel Approach to Optimal Price-Setting of Time-and-Level-of-Use Tariffs}}
\author{
	\authorname{Mathieu Besançon}\\
	\authorgroup{Department of Mathematics \& Industrial Eng., Polytechnique Montréal \& INRIA Lille \& Centrale Lille}\\
	\authoraddress{INRIA Lille, 40 avenue Halley, 59650 Villeneuve-d'Ascq, France}\\
	\authormail{Email: mathieu.besancon@polymtl.ca}\\\\
	\authorname{Miguel F. Anjos*}\\
	\authorgroup{School of Mathematics, University of Edinburgh}\\
	\authoraddress{James Clerk Maxwell Building, Peter Guthrie Tait Road, Edinburgh EH9 3FD, UK}\\
	\authormail{Email: anjos@stanfordalumni.org}\\\\
	\authorname{Luce Brotcorne}\\
	\authorgroup{INOCS, INRIA Lille}\\
	\authoraddress{INRIA Lille, 40 avenue Halley, 59650 Villeneuve-d'Ascq, France}\\
	\authormail{Email: luce.brotcorne@inria.fr}\\\\
	\authorname{Juan A. Gomez-Herrera}\\
	\authorgroup{GERAD \& Department of Mathematics \& Industrial Eng., Polytechnique Montréal}\\
	\authoraddress{2900 Boulevard Edouard-Montpetit, Montréal, QC H3T1J4, Canada}\\
	\authormail{Email: juan.gomez@polymtl.ca}
}
\date{}

\begin{document}

%Abstract goes here--------------------------------------------------------------

\twocolumn[

\maketitle

\thispagestyle{fancy} %Insert header in title page

\begin{center}
\line(1,0){465}
\end{center}

\begin{@twocolumnfalse}

\textbf{Summary}
\\

Time-and-Level-of-Use (TLOU) is a recently proposed pricing policy for energy,
extending Time-of-Use with the addition of a capacity that users can book for a given time frame,
reducing their expected energy cost if they respect this self-determined capacity limit.
We introduce a variant of the TLOU defined in the literature, aligned with the
supplier interest to prevent unplanned over-consumption.
The optimal price-setting problem of TLOU is defined as a bilevel, bi-objective
problem anticipating user choices in the supplier decision. An efficient
resolution scheme is developed, based on the specific discrete structure of the
lower-level user problem. Computational experiments using consumption
distributions estimated from historical data illustrate the effectiveness
of the proposed framework.

\begin{center}
\line(1,0){465}
\end{center}

\end{@twocolumnfalse}

]

%--------------------------------------------------------------------------------

\section{Introduction}
\blfootnote{This extended abstract was presented at, and included in the program of, the 13th EUROGEN conference on Evolutionary and Deterministic Computing for Industrial Applications.}
The increasing penetration of wind and solar energy has marked the last decades,
bringing a decentralization and higher stochasticity of power generation,
yielding new challenges for the operation of electrical grids.
Advances in communication technologies have enabled both scheduling of smart
appliances and seamless data collection and exchange between different entities
operating on power networks, from generators to end-consumers.
Using such capabilities, agents can make decisions based on probability
distributions of different variables, possibly conditioned on known external
factors, such as the weather. \\

As a promising lead to tackle these challenges, Demand Response (DR) has
attracted an increasing interest in the past decades. Instead of compensating
ever-increasing fluctuations of renewable production with conventional
power generation, this approach consists in leveraging the flexibility of some
consuming units by reducing or shifting their load\cite{Albadi2008}.
The Time-and-Level-of-Use system presented in this paper is primarily a
price-based Demand Response program based on the classification found in the
literature\cite{Albadi2008}, although the self-determined consumption limit
creates an incentive for respecting the commitment on the part of the user.
Its design makes opting out of the program a natural extension
and thus requires less coupling and interaction between users and suppliers to
work. In a 2012 report, the US Federal Energy Regulatory Commission identified
the lack of short-term estimation methods as one of the critical barriers to the
effective implementation of Demand Response \cite{ferc2012}.
In a related approach\cite{vuelvas18game}, an incentive-based Demand Response
program is developed with users signaling both a predicted consumption
and a reduction capacity, with the supplier selecting reductions randomly.
Both the program developed by the authors\cite{vuelvas18game} and the TLOU policy
require a signal sent from the user to the supplier.
\\

Bilevel optimization has been used to model and tackle optimization problems
in energy networks \cite{dempe15}. Applications on power systems and markets are
also mentioned in several reviews on bilevel optimization \cite{Colson2007,shi-xia}.
It allows decision-makers to encapsulate utility-generators, user-utility
interactions \cite{generation1,generation2}, or define robust
formulations of unit commitment and optimal power flow \cite{robusttri}.
Some recent work focuses on bilevel optimization for demand-response in the
real-time market \cite{dr-realtime}.
Multiple supplier settings are also considered, for instance in the
determination of Time-of-Use pricing policies \cite{customizing18},
where competing retailers target residential users.

The Time-and-Level-of-Use (TLOU) pricing scheme \cite{gomez17} was built as an
extension of Time-of-Use (TOU), targeting specifically the issue
with current large-scale Demand-Response programs identified in the FERC report\cite{ferc2012}.
The optimal planning and operation of a
smart building under this scheme was developed, taking the TLOU settings as
input of the decision. In this work, we consider the perspective of the supplier
determining the optimal parameters of TLOU.\\

We integrate the optimal user reaction in the supplier decision problem, thus
modelling the interaction as a Stackelberg game solved as a bilevel
optimization problem. The specific structure of the lower-level decision is
leveraged to reduce the set of possibly optimal solutions to a finite enumerable
set. Through this reduction, the lower-level optimality conditions are translated
into a set of linear constraints. The set of optimal pricing configurations is
derived for different distributions corresponding to different time frames.
These options can be computed in advance and then one is chosen by the supplier ahead
of the consumption time to create an incentive to book and consume a given
capacity. We show the effectiveness of the method with discrete probability
distributions computed from historical consumption data.

\section{Time-and-Level-of-Use pricing}

The TLOU policy was initially developed from the perspective of smart consumption
units\cite{gomez17} which can be a building, an apartment or a micro-grid monitoring
its consumption and equipped with programmable consuming devices.
It is a pricing model for energy built upon the Time-of-Use
(TOU) implemented by several jurisdictions and for which the energy price
is fixed by intervals throughout the day. TLOU extends TOU by allowing users to
self-determine and book an energy capacity at each time frame
depending on their planned requirements; and by doing so, they provide the supplier
with information on the intended consumption.
We will refer to the capacity as the amount of energy booked by the user for a
given time frame, following the same terminology as the reference implementation\cite{gomez17}.
Energy prices still depend on the
time frame within the day, but also on the capacity booked by the user.
This pricing scheme is applied in a three-phase process:
\begin{enumerate}
	\item The supplier sends the pricing information to the user.
	\item The user books a capacity from the supplier for the time frame before a given deadline. If no capacity has been booked, the Time-of-Use pricing is used.
	\item After the time frame, the energy cost is computed depending on the energy consumed $x$ and booked capacity $c$:
	\begin{itemize}
		\item If $x \leq c$, then the applied price of energy is $\pi^L(c)$ and the energy cost is $\pi^L(c) \cdot x$.
		\item If $x > c$, then the applied price of energy is $\pi^H(c)$ and the energy cost is $\pi^H(c) \cdot x$.
	\end{itemize}
\end{enumerate}
In the model considered in this paper, only the first two steps involve decisions from the agents,
making the decision process a Stackelberg game given the sequentiality of these decisions.
The price structure is composed of three elements: a booking fee $K$,
a step-wise decreasing function $\pi^L(c)$ representing the lower energy price and
a step-wise increasing function $\pi^H(c)$ representing the higher energy price.
The steps of the lower and higher price functions are given at different breakpoints:
\begin{align*}
&\{c^L_0, c^L_1, c^L_2, ...\} = C^L &\&&\,\, \{\pi^L_0, \pi^L_1, \pi^L_2, ...\} = \pi^L \\
&\{c^H_0, c^H_1, c^H_2, ...\} = C^H &\&&\,\, \{\pi^H_0, \pi^H_1, \pi^H_2, ...\} = \pi^H
\end{align*}

\noindent
$\pi^L(c)$ will refer to the function of the capacity and $\pi^L_j$ to
the value of the lower price at step $j$. In the initial version of the pricing
\cite{gomez17}, the energy consumed above the capacity
is paid at the higher tariff while the rest is paid at the lower tariff.\\

Considering that most power systems try to prevent over-consumption and
unplanned excess consumption, we introduce a variant where the whole energy
consumed is paid at the lower tariff if it is less or equal to the capacity, and
at the higher tariff otherwise, see Equation (\ref{eq:tlouprice}).
In other words, if the
consumption over the time frame remains below the booked capacity, the effective
energy price is given by the lower tariff curve; if the consumption exceeds
the booked capacity, the energy price is given by the higher tariff.
The total cost for the user associated with a booked capacity $c$ and a
consumption $X_t$ for a time frame $t$ is:

\begin{equation}
\mathcal{C}(c_t; X_t) = \begin{cases} K \cdot c_t + \pi^L(c_t) \cdot X_t, & \mbox{if} \, X_t\ \leq \ c_t, \\
K \cdot c_t + \pi^H(c_t) \cdot X_t & \mbox{otherwise.}\end{cases}\label{eq:tlouprice}
\end{equation}

%\nomenclature[V]{$K$}{Booking fee ($\$/kWh$)}
%\nomenclature[P]{$\pi^0(t)$}{Baseline price at time $t$ ($\$/kWh$)}
%\nomenclature[V]{$c$}{Capacity booked by the user ($kWh$)}
%\nomenclature[V]{$\pi^L(c)$}{Lower tariff for capacity $c$ ($\$/kWh$)}
%\nomenclature[V]{$\pi^H(c)$}{Higher tariff for capacity $c$ ($\$/kWh$)}
%\nomenclature[S]{$S$}{Set of candidate solutions ($kWh$)}
%\nomenclature[P]{$c_k,\ k \in S$}{$k$-th capacity from the set of candidate solutions}
%\nomenclature[P]{$X_t$}{Total energy consumed within the time frame $t$}
%\nomenclature[F]{$\mathcal{C}(c;x)$}{User cost with booked capacity $c$ and consumption $x$}
%\nomenclature[F]{$\mathcal{C}(c)$}{Expected user cost over all scenarios for capacity $c$}
%\nomenclature[F]{$L^F(c,K,\pi^L,\pi^H)$}{Supplier financial loss function}
%\nomenclature[F]{$L^G(c)$}{Supplier guarantee loss function}

The load distribution $X_t$ is a random variable; both user and supplier make
decisions on the expected cost over the set of possible consumption levels
$\Omega$, given as:
\begin{equation}
\mathcal{C}(c_t) = \mathbb{E}_{\Omega}\left[\mathcal{C}(c_t; x_{\omega})\right]
\end{equation}
where $\mathbb{E}_{\Omega}$ is the expected value over the
support $\Omega$ of the probability distribution. \\

Furthermore, $\pi^L(c = 0) = \pi^H(c = 0) = \pi_0(t)$, with $\pi_0(t)$ the
Time-of-Use price at the time frame of interest $t$. This property allows users
to opt-out of the program for some time frames by simply not booking any capacity.
TLOU is designed for the day-ahead market, where both the pricing components and the capacity are
chosen ahead of the consumption time \cite{gomez17}. It can however be adjusted
to other markets \cite{germanpower14} or intra-day settings.
The entity defining the TLOU pricing can also extend the possible settings by
decoupling the time spans for one price setting choice from the time frames for
capacity booking. For instance, a price setting can be chosen by the utility
for the week, while the booking of capacity occurs the day before the
consumption. \\

TLOU offers the user the possibility to reduce their cost of energy by load
planning, and offers the supplier the prospect of improved load forecasting,
because under-consumption is paid by the excess booking cost while
over-consumption is paid by the difference between higher and lower tariffs.

We illustrate this phenomenon with
a supplier decision $(K,\pi^L,\pi^H)$ on Figures
\ref{fig:game1} and \ref{fig:game2} using the relative cost:
\begin{equation}
x \mapsto \frac{\mathcal{C}(c;x)}{x}
\end{equation}
for different values of $c$.
For $c > 0$, the fixed cost $K\cdot c$ makes low consumption levels
expensive per consumed unit, while the transition from lower to
higher price makes over-consumption more expensive than the baseline.
In the example, $c=1.5$ cannot be an optimal booked capacity
since it is always more expensive than the baseline for all
possible consumption levels.
The cases $c = 3.0$ and $c = 3.5$ both have ranges of consumption
for which this choice is optimal for the user;
these ranges always have the capacity as upper bound.
The difference in cost occurring at the transition from lower to higher
price increases with the booked capacity, illustrating the guarantee
against over-consumption the supplier gains from the capacity booked
by the user. Regardless of the ability to shift consumption or irrational
decision-making, the commitment to a capacity creates a cost difference if the
user is not making an optimal decision, which compensates the supplier for
these unexpected deviations. This financial incentive against a consumption
above the booked capacity creates a guarantee of interest to the supplier
and is cast as a second objective expressed in Section
\ref{sec:bilevel-model}.
\\

\begin{figure}[ht]
\begin{minipage}{0.49\textwidth}
    \centering
    \includegraphics[width=1.15\textwidth]{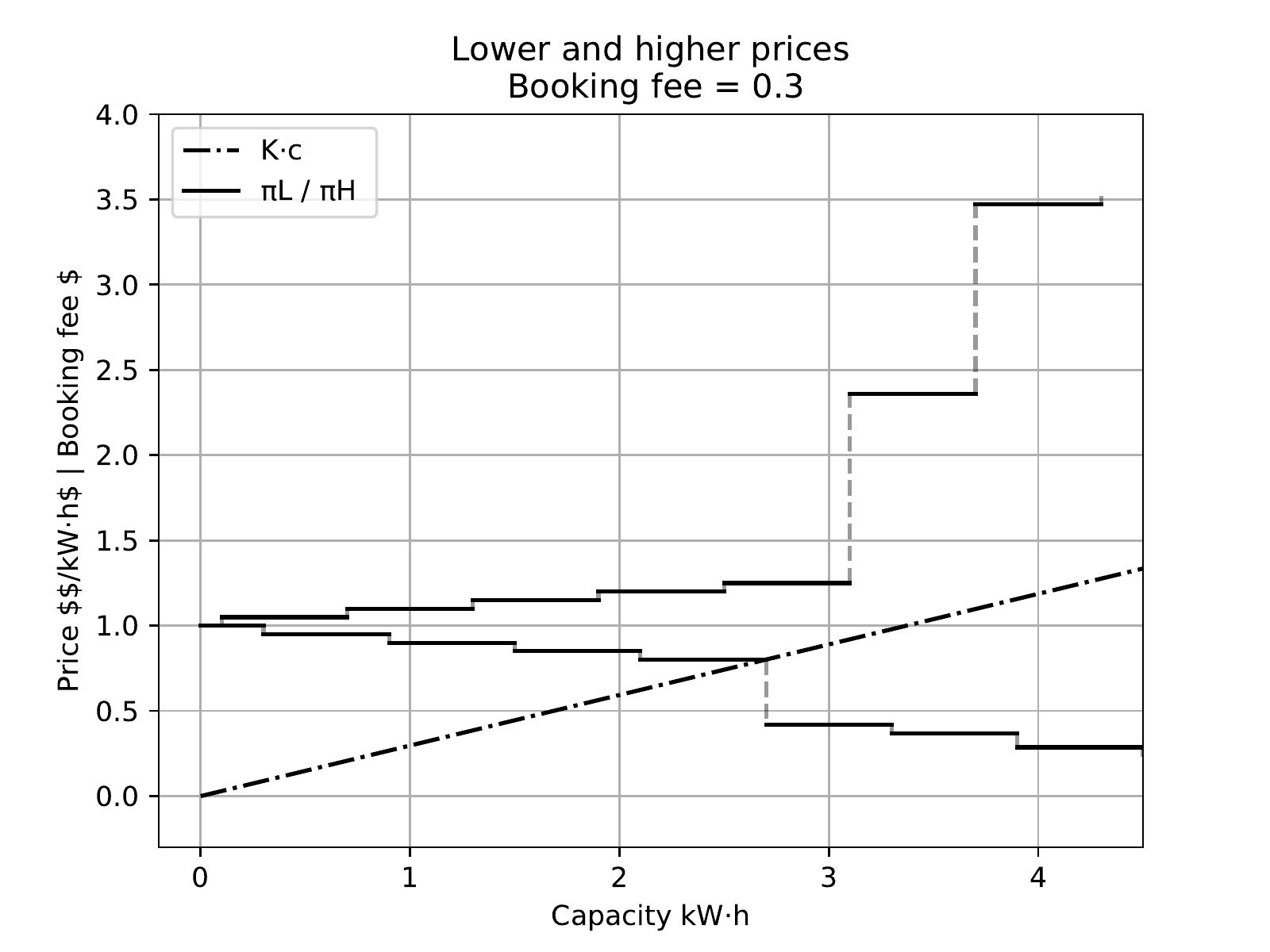}
\caption{Example of TLOU pricing}
\label{fig:game1}
\end{minipage}
\begin{minipage}{0.49\textwidth}
    \centering
    \includegraphics[width=1.15\textwidth]{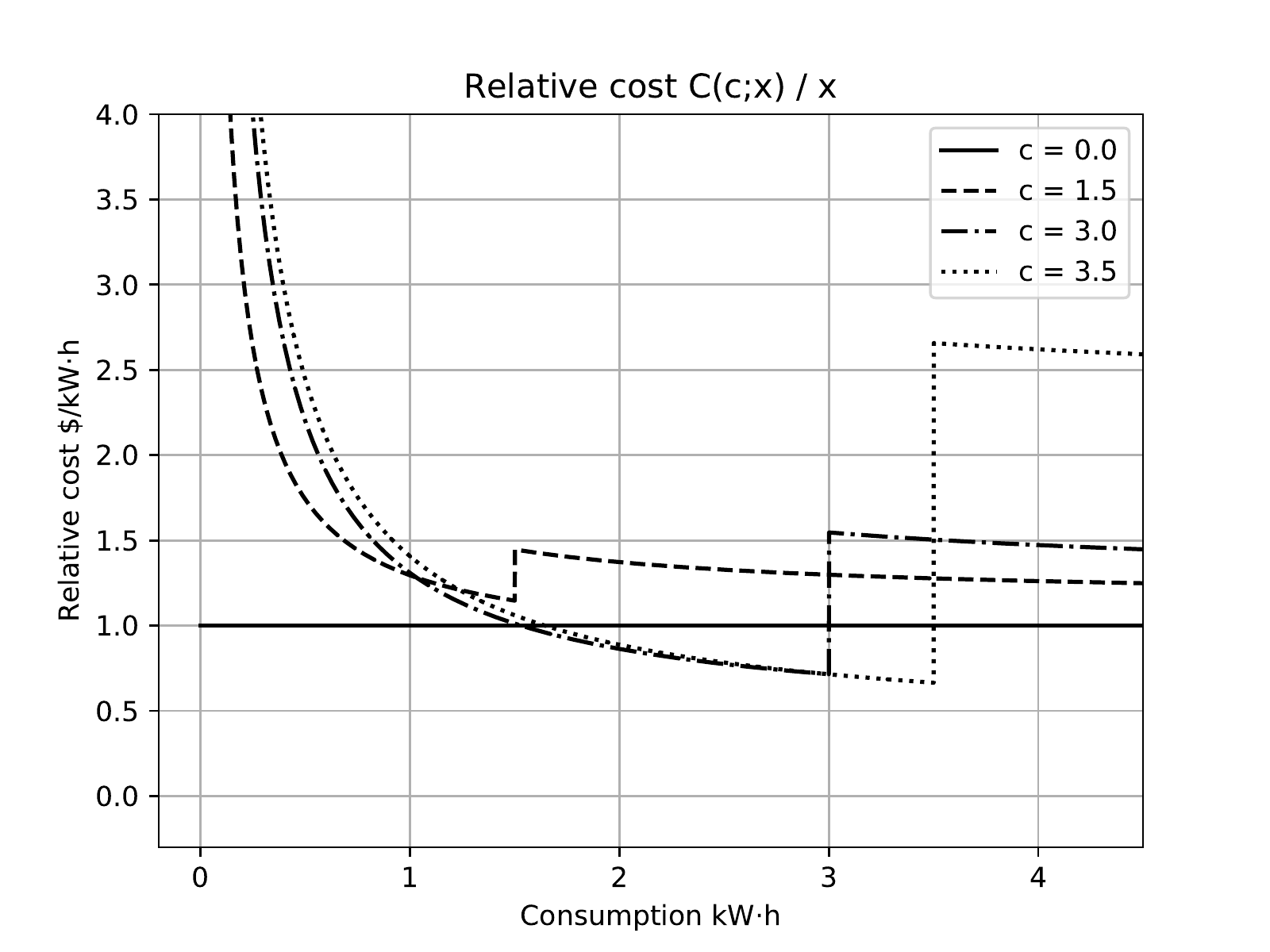}
    \caption{Relative cost of energy vs consumption for different capacities}\label{fig:game2}
\end{minipage}
\end{figure}

The supplier first builds their set of options based on prior consumption data.
In the proposed method, the prior discrete distribution used can be conditioned
on some independent variables if they are known and influence the consumption
(e.g. forecast external temperature or day of the week).
They can then pick a pricing setting for a given day based on generation-side
considerations and constraints, including the option to stay at a flat
Time-of-Use tariff for some or all time frames.

\section{Bilevel model of the supplier problem}\label{sec:bilevel-model}

The supplier wishes to determine an optimal set of pricing options at any
capacity level. In the model developed in this section, we consider a discrete
probability distribution with a finite support, derive some properties
of the cost structure which we then leverage to formulate the optimization
problem in a tractable form.\\

At any capacity level, the decision process of the supplier involves two
objectives, the expected revenue from the tariff and the guarantee of an
upper bound on the consumption. The expected revenue is given by:
\begin{equation}
\mathcal{C}(c) = K\cdot c + \sum_{\omega\in\Omega^-(c)} x_{\omega} p_{\omega} \pi^L(c) + \sum_{\omega\in\Omega^+(c)} x_{\omega} p_{\omega} \pi^H(c)
\end{equation}

\noindent
with any capacity booked defining a partition of the set of scenarios:
\begin{gather}
\Omega^-(c) = \{\omega \in \Omega, x_{\omega} \leq c\}\\
\Omega^+(c) = \{\omega \in \Omega, x_{\omega} > c\}
\end{gather}

The function $\mathcal{C}(c)$ is minimized by the user with respect to their decision $c$.
It is non-linear, non-smooth and discontinuous because of the partition of the
scenarios by $c$ and the transition between steps of the pricing curves
$\pi^L(c)$, $\pi^H(c)$. Both the user and supplier problems are thus
intractable with this initial formulation. Proposition \ref{prop:subset} shows
that only a discrete finite subset of capacity values are
candidates to optimality for the user.

\begin{proposition}\label{prop:subset}
The optimal booked capacity for a user at a time frame $t$
belongs to a discrete and finite set of capacities $S_t$, defined as:
\begin{equation}
\centering
S_t = \{0\} \cup C^L \cup \Omega_t
\end{equation}
with $\Omega$ the set of consumption scenarios.
\end{proposition}

\begin{proof}
The user objective function is the sum of the booking cost and the expected
electricity cost. The booking cost is linear in the booked capacity,
with a positive slope equal to the booking fee.
The expected electricity cost is piecewise constant in the
booked capacity, with discontinuities at steps of both of
the price curves because of the $\pi^L$ and $\pi^H$ prices and at
possible load levels because of the transfer of a load from
$\Omega^+$ to the $\Omega^-$ set. This can be highlighted using the
indicator functions associated with each of the two sets:
\begin{align}
&\mathbbm{1}^-(\omega,c) = \begin{cases}
1, & \mbox{if $x_{\omega} \leq c$,} \\ 0 & \mbox{otherwise.}
\end{cases}\\
&\mathbbm{1}^+(\omega,c) = 1 - \mathbbm{1}^-(\omega,c)
\end{align}
The expression of the user expected cost becomes:
\begin{multline}\label{eq:indicator-expected}
\mathcal{C}(c) =  K \cdot c + \\ \sum_{\omega \in \Omega} x_{\omega} \cdot p_{\omega} \cdot \left(\pi^L(c) \cdot \mathbbm{1}^-(\omega,c) + \pi^H(c) \cdot \mathbbm{1}^+(\omega,c) \right)
\end{multline}

The sum of the two terms is therefore piecewise linear with a
positive slope. On any interval between the discontinuity points,
the optimal value lies on the lower bound, which can be any point
of $C^L$, $C^H$, $\Omega$ or 0.\\

Furthermore, let $\mathcal{C}(c)$ be the user cost for a booked
capacity and $\bar{c}$ such that $\bar{c} \in C^H$ and
$\bar{c} \notin \{0\} \cup C^L \cup \Omega$.
The higher tariff levels are monotonically increasing. Let
$\varepsilon > 0 $ be sufficiently small such that:
\begin{align*}
\pi^H(\bar{c} - \varepsilon) &= \pi_{n}, \\
\pi^H(\bar{c} + \varepsilon) &= \pi_{n+1} > \pi_{n}, \\
\pi^L(\bar{c} - \varepsilon) &= \pi^L(\bar{c} + \varepsilon) = \pi^L_{m}, \\
\nexists x_{\omega},\ \omega \in &\Omega \text{ s.t. } \bar{c} - \varepsilon \leq x_{\omega} \leq \bar{c} + \varepsilon.
\end{align*}
The last condition guarantees there is no load value in the
$\left[\bar{c} - \varepsilon, \bar{c} + \varepsilon\right]$ interval and can also be expressed in terms of the two load sets split by the capacity:
\begin{gather*}
\Omega^+(\bar{c} - \varepsilon) = \Omega^+(\bar{c} + \varepsilon)\ \text{and}\ \
\Omega^-(\bar{c} - \varepsilon) = \Omega^-(\bar{c} + \varepsilon).
\end{gather*}
Then if such $\varepsilon$ exists, we find that:
\begin{align*}
&\mathcal{C}(\bar{c} - \varepsilon) = K \cdot (\bar{c} - \varepsilon) + \sum_{\omega \in \Omega^-(\bar{c})} \pi^L_m \cdot x_{\omega} + \sum_{x \in \Omega^+(\bar{c})} \pi^H_{n} \cdot x_{\omega},  \\
&\mathcal{C}(\bar{c} + \varepsilon) = K \cdot (\bar{c} + \varepsilon) + \sum_{\omega \in \Omega^-(\bar{c})} \pi^L_m \cdot x_{\omega} + \sum_{\omega \in \Omega^+(\bar{c})} \pi^H_{n+1} \cdot x_{\omega},  \\
&\mathcal{C}(\bar{c} + \varepsilon) - \mathcal{C}(\bar{c} - \varepsilon) = 2 \varepsilon K + \sum_{\omega \in \Omega^+(\bar{c})} (\pi^H_{n+1}-\pi^H_n) \cdot x_{\omega}, \\
&\mathcal{C}(\bar{c} + \varepsilon) - \mathcal{C}(\bar{c} - \varepsilon) > 0.
\end{align*}

The discontinuity on any $x \in C^H$ is therefore always positive
and cannot be a candidate for optimality. It follows that optimality candidates
are restricted to the set $S = \{0\} \cup C^L \cup \Omega$.
\end{proof}

Proposition \ref{prop:subset} means we can replace the continuous decision
set of capacities with a discrete set that can be enumerated.\\

The guarantee of an upper bound on the consumption $\mathcal{G}$ corresponds to the
incentive given to the user against consuming above the considered
capacity. It is the second objective, given by the difference in cost at
the capacity, which is the immediate difference in total cost at the
transition from lower to higher tariff:
\begin{equation}
\mathcal{G}(c, \pi^L, \pi^H) = c_t \cdot \left(\pi^H(c) - \pi^L(c)\right).
\end{equation}

The supplier needs to include the user behavior and optimal reaction in their
decision-making process, which can be done by a bilevel constraint:
\begin{equation}
c_t \in arg\min_{c} \mathcal{C}(c).
\end{equation}

The user thus books the least-cost option at each time frame, given the
corresponding probability distribution. Given the finite set of optimal
candidates $S_t$ defined in \ref{prop:subset}, this constraint can be
re-written as:
\begin{equation}
\mathcal{C}(c_t) \leq \mathcal{C}(c) \,\, \forall c \in S_t.
\end{equation}

If multiple choices of $c$ yield the minimum cost, the choice of
the user is not well-defined. The supplier would want to ensure the uniqueness
of the preferred solution by making it lower than the expected cost of any other
capacity candidate by a fixed quantity $\delta > 0$. This quantity can be
interpreted as the conservativeness of the user (unwillingness to move to an
optimal solution up to a difference of $\delta$). It is a parameter of the
decision-making process, estimated by the supplier. The lower-level optimality
constraint is then for a preferred candidate $k$:
\begin{equation}\label{eq:lowerineq}
\mathcal{C}(c_{tk}, K, \pi^L,\pi^H) \leq \mathcal{C}(c_{tl}, K, \pi^L,\pi^H) - \delta \,\,\forall l \in S_t \backslash k.
\end{equation}

\subsection{Additional contractual constraints}

In order to obtain regular price steps, the contract between supplier and
consumer can include further constraints on the space of pricing parameters.
We include three types of constraints: lower and upper bounds on the booking
fee $K$, minimum and maximum increase at each step of the higher price and
minimum and maximum decrease at each step of the higher price.
All these can be expressed as linear constraints, and we
gather them under the constraint set:
\begin{equation}
(K,\pi^L,\pi^H) \in \Phi.
\end{equation}

\subsection{Complete optimization model}

The model is defined for each of the capacity candidates and is thus noted
$\mathcal{P}_{tk}$ for candidate $k$ and time frame $t$:

\begin{align}
\max_{K,\pi^L,\pi^H} & \left(\mathcal{C}(c_{tk}, K, \pi^L, \pi^H), \mathcal{G}(c_{tk}, \pi^L, \pi^H)\right)\\
& (K,\pi^L,\pi^H) \in \Phi \\
& \mathcal{C}(c_{tk}, K, \pi^L, \pi^H) \leq \mathcal{C}(c_{lt}, K, \pi^L, \pi^H) - \delta \,\, \forall l \in S_t \backslash k,
\end{align}

\noindent where
\begin{align}
\mathcal{C}(c_{tk}, K, \pi^L, \pi^H) &= \nonumber\\
K\cdot c_{tk} &+ \sum_{\omega \in \Omega^-_t(c_{tk})} x_{\omega} p{\omega} \pi^L_k + \sum_{\omega \in \Omega^+_t(c_{tk})} x_{\omega} p_{\omega} \pi^H_k, \\
\mathcal{G}(c_{tk}, \pi^L, \pi^H) &= c_{tk} \cdot \left(\pi^H(c) - \pi^L(c)\right).
\end{align}

\section{Solution method and computational experiments}

The proposed model was implemented and tested
using historical consumption data\cite{ucirepo} measured on a pilot house.
The instantaneous consumption was measured
every two minutes during 47 months on a residential
building by the energy supplier and grid operator EDF\cite{consumptiondata}.
Since the focus is the energy consumed
within a given time frame, the instantaneous power can be
averaged for each hour, yielding the energy in $kW\cdot h$ and
avoiding issues of missing measurements in the dataset.
%The average of the hourly consumption over all days of
%the data set was computed, consumption peaks can be observed
%in the early morning (7AM, 8AM) and evening (19AM - 21AM).
Data preprocessing, density estimation and discretization, and visualization
were performed using Julia \cite{julia101}, matplotlib \cite{Hunter:2007}
and \textit{KernelDensity.jl} \cite{kerneldensity}.
The construction and optimization of the model were carried out using CLP from
the COIN-OR project \cite{coinpaper} as a linear solver and JuMP
\cite{DunningHuchetteLubin2017}.
For all experiments where it is not specified, an inertia of $\delta=0.05$ has been applied.
For every time frame and for all capacity candidates, the bi-objective supplier
problem with objectives $(\mathcal{C}, \mathcal{G})$ is solved with the
$\varepsilon$-constraint method implemented in
\textit{MultiJuMP.jl}\cite{asbjorn_nilsen_riseth_2019_2565839}.
In all cases, the objectives are found to be non-conflicting, in the sense
that the utopia point of the multi-objective problem is feasible and reached.
This implies that a lexicographic multi-objective optimization solves the problem
and reaches the utopia point, but does not guarantee that this holds for all
problem configurations.
\\

Figure \ref{fig:numoptions} shows the number of options computed at each
hourly time frame and Figure \ref{fig:levoptions} the capacity level in
$kW\cdot h$ of each option. The option of a capacity level of 0 is always possible.
Two examples of TLOU settings obtained are presented in Figures
\ref{fig:solution14} and \ref{fig:solution18}.

\begin{figure}[ht]
\begin{minipage}{0.49\textwidth}
    \centering
    \includegraphics[width=1.1\textwidth]{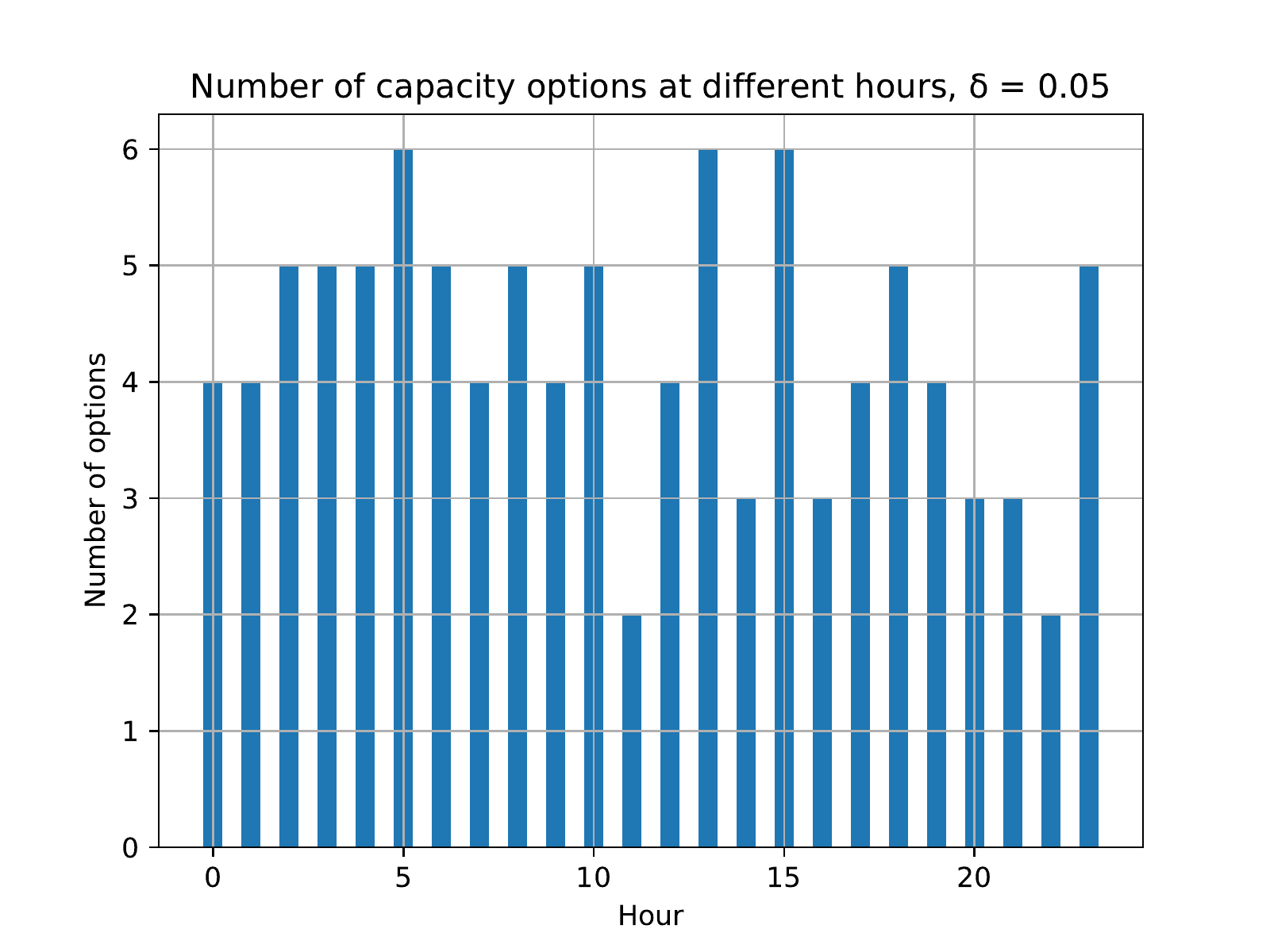}
\caption{Number of options at different hours}
\label{fig:numoptions}
\end{minipage}
\begin{minipage}{0.49\textwidth}
    \centering
    \includegraphics[width=1.1\textwidth]{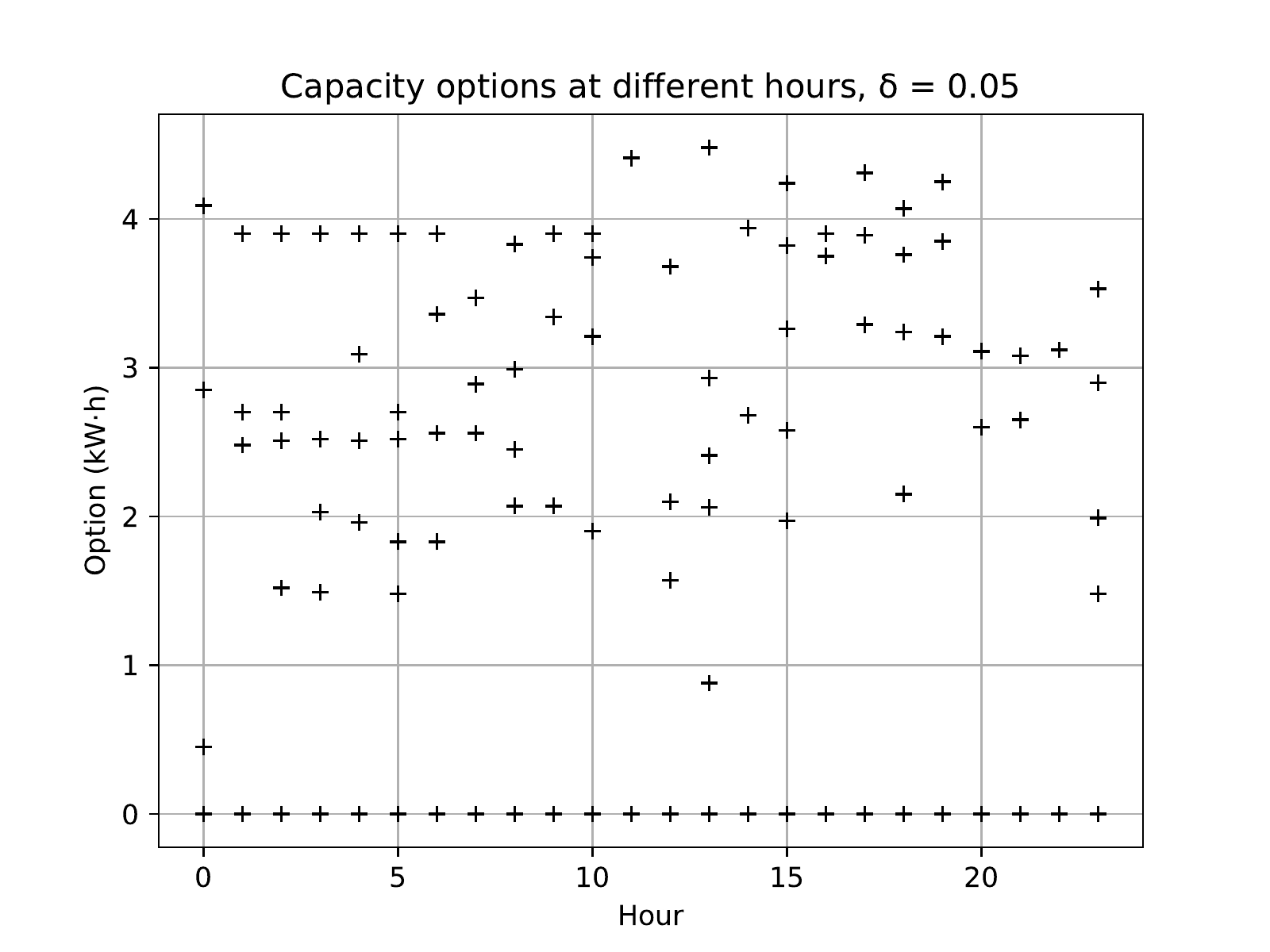}
    \caption{Capacity levels of options at different hours}\label{fig:levoptions}
\end{minipage}
\end{figure}

\begin{figure}[ht]
\begin{minipage}{0.47\textwidth}
    \centering
    \includegraphics[width=1.05\textwidth]{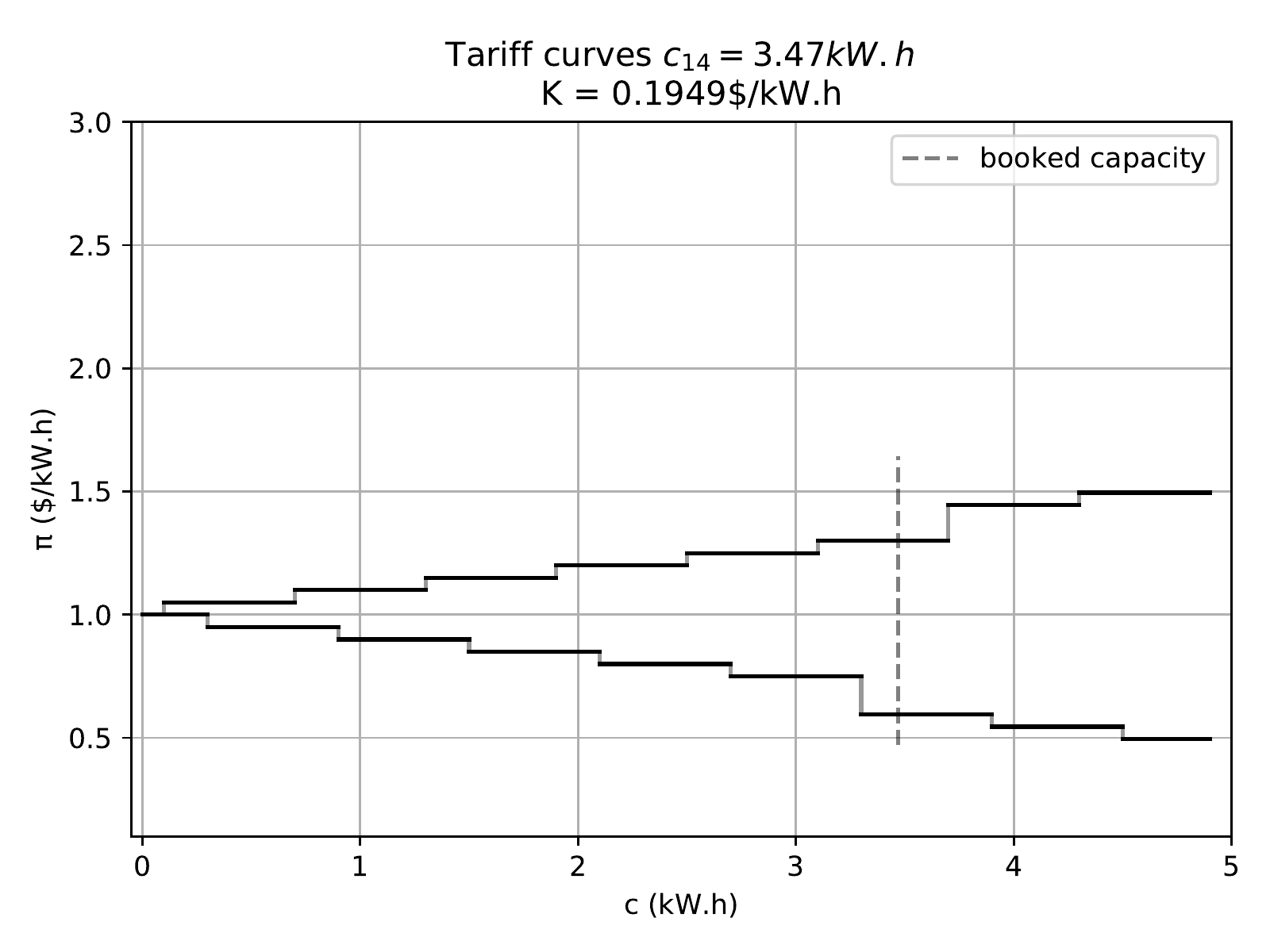}
\caption{TLOU price settings to incentive for capacity $c=3.47kW\cdot h$}
\label{fig:solution14}
\end{minipage}
\begin{minipage}{0.47\textwidth}
    \centering
    \includegraphics[width=1.05\textwidth]{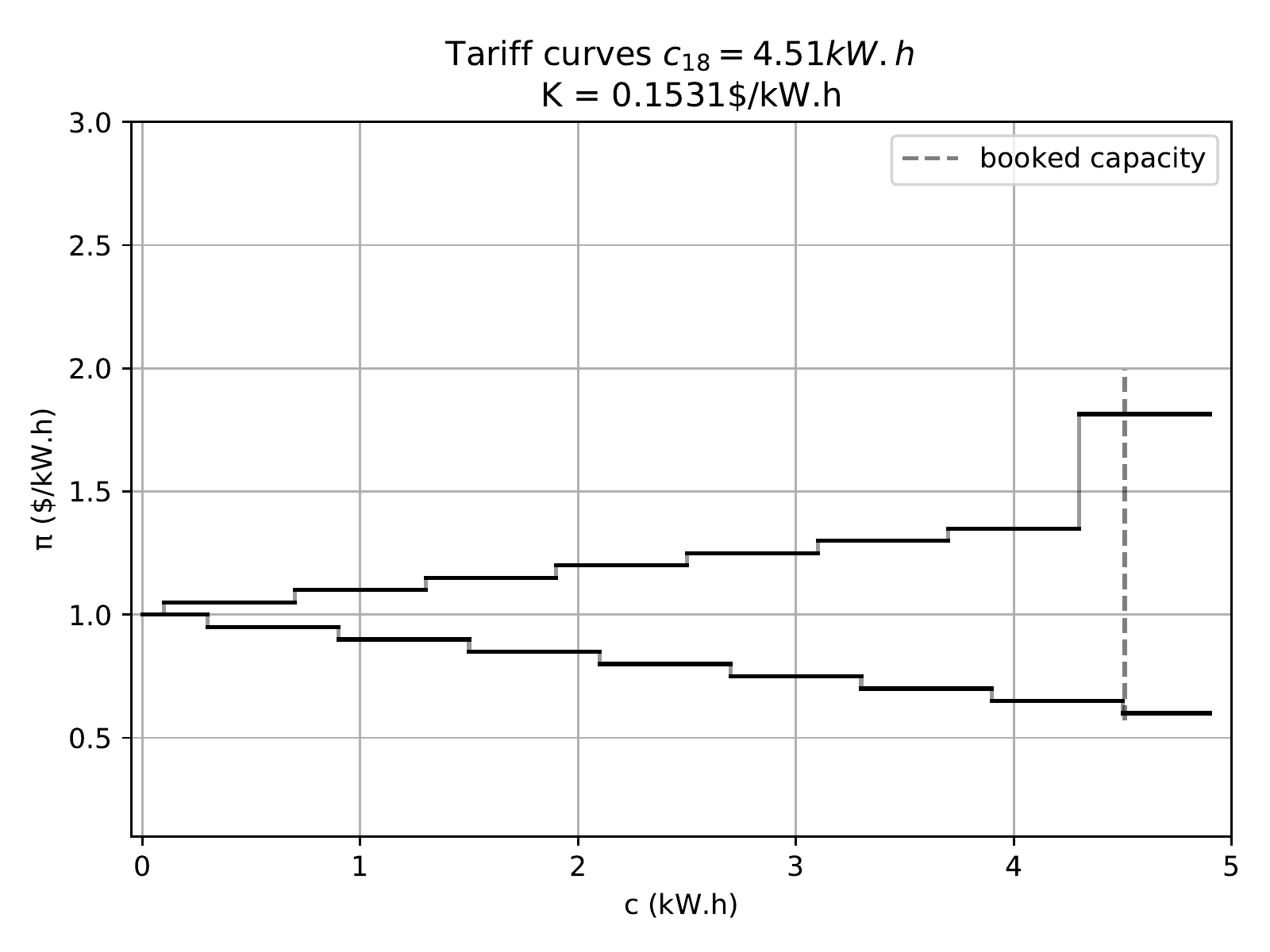}
    \caption{TLOU price settings to incentive for capacity $c=4.51kW\cdot h$}
		\label{fig:solution18}
\end{minipage}
\end{figure}

The expected cost for the user of booking any capacity, given the price setting
provided in Figure \ref{fig:solution14} is shown Figure \ref{fig:exp_cost14}.
The most notable result is that in
all cases tested, the utopia point, defined as the optimal value of the two
objectives optimized separately, is reachable.
This result is conceivable given that the supplier decision is taken in a
high-level space, allowing multiple solutions to be optimal with respect to the
revenue. In order to ensure the guarantee-maximizing optimal solution, a two-step
process lexicographic multi-optimization procedure is used:
\begin{enumerate}
	\item Solve the revenue-maximizing problem to obtain the maximal reachable revenue $v$.
	\item Solve the guarantee-maximizing problem, while constraining a revenue $\mathcal{C}(\cdot) \geq v$.
\end{enumerate}

Figures \ref{fig:base_solution1} and \ref{fig:compare_solution2} present a
TLOU configuration optimized for costs only and for cost and guarantee.

\begin{figure}[ht]
\centering
\includegraphics[width=0.49\textwidth]{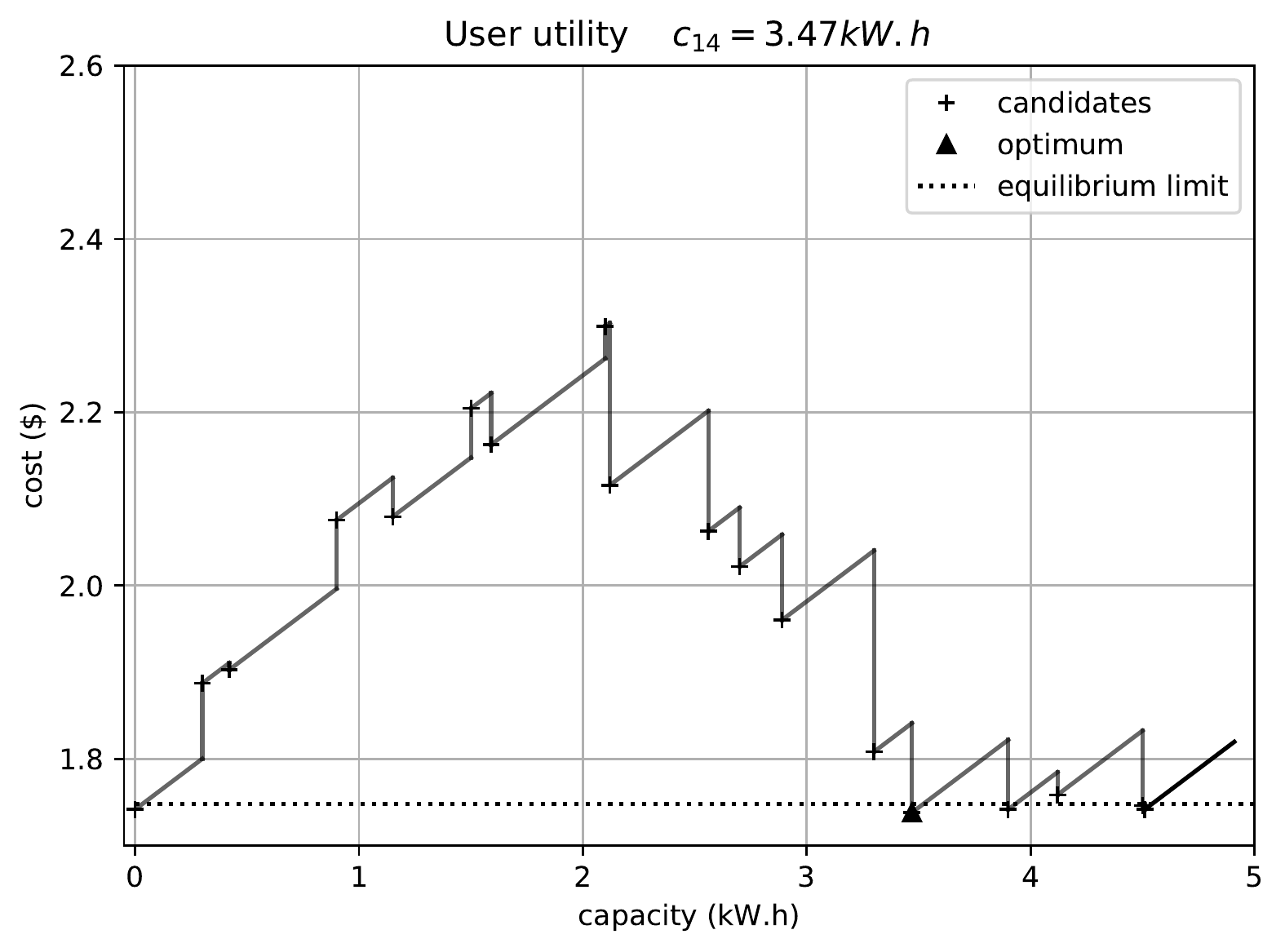}
\caption{Expected cost for any capacity booked under the price settings of Figure \ref{fig:solution14}}
\label{fig:exp_cost14}
\end{figure}
\begin{figure}[ht]
\begin{minipage}{0.495\textwidth}
    \centering
    \includegraphics[width=1.05\textwidth]{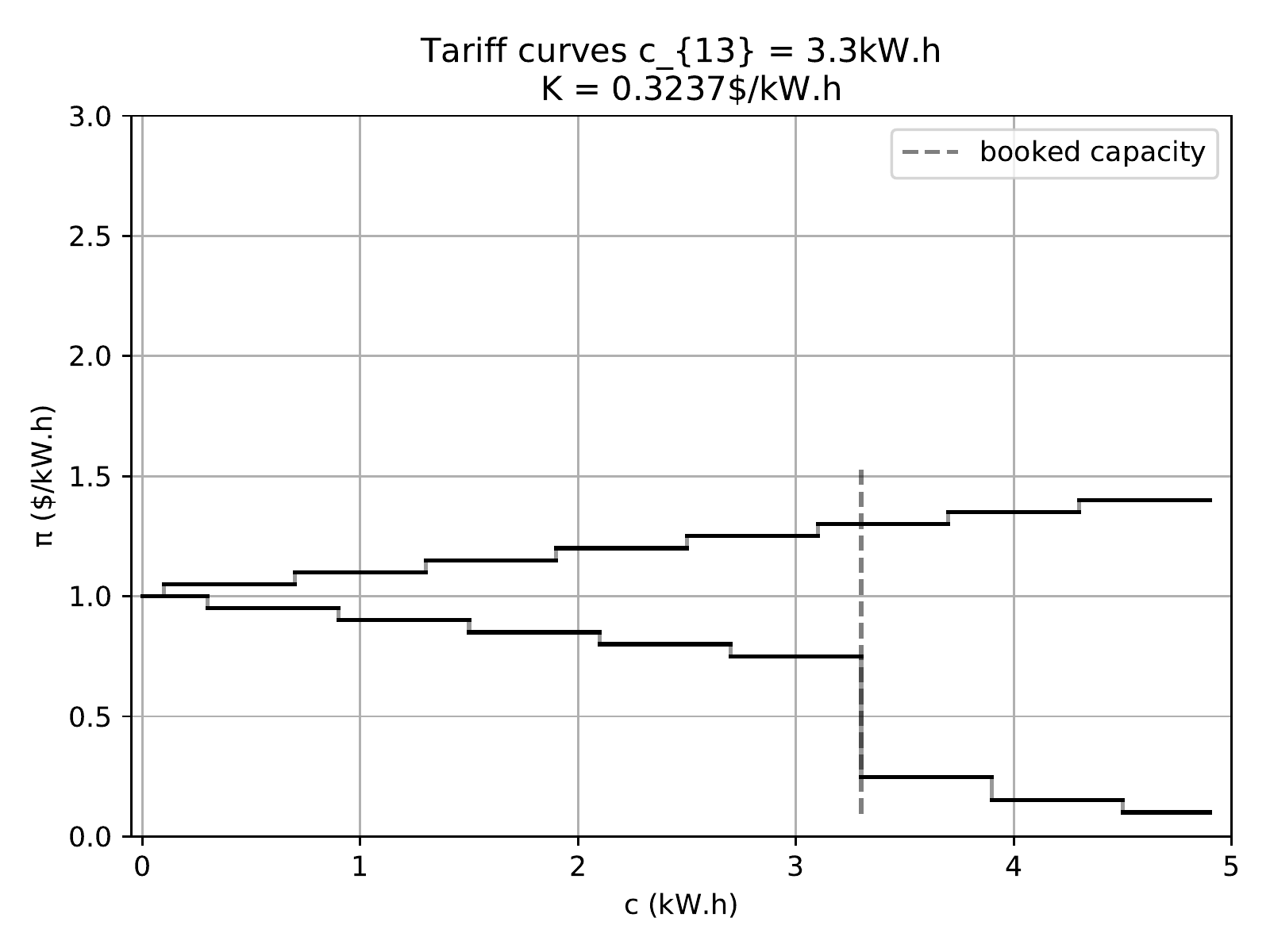}
\caption{TLOU price settings for revenue maximization only}
\label{fig:base_solution1}
\end{minipage}
\begin{minipage}{0.495\textwidth}
    \centering
    \includegraphics[width=1.05\textwidth]{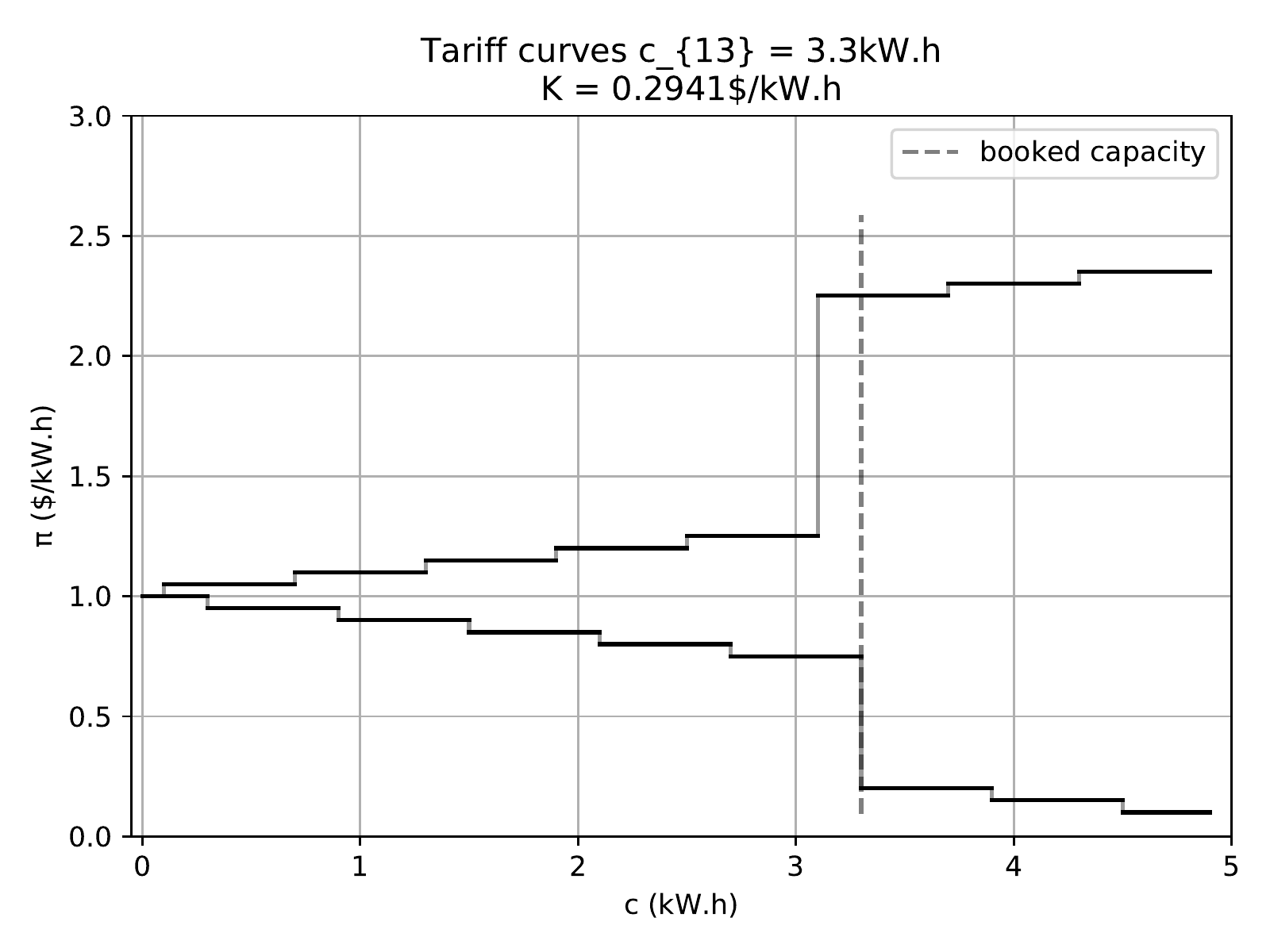}
    \caption{TLOU price settings after lexicographic optimization}
		\label{fig:compare_solution2}
\end{minipage}
\end{figure}

All the models solved are linear optimization problems with a fixed number of variables
and a number of constraints growing linearly with the number of scenarios
considered.
However, all these constraints are of type:
\begin{equation}
\mathcal{C}(c_{tk}, K, \pi^L,\pi^H) \leq \mathcal{C}(c_{tl}, K, \pi^L,\pi^H) - \delta \,\,\forall l \in S_t \backslash k.
\end{equation}

\noindent This is equivalent to:
\begin{equation}
\mathcal{C}(c_{tk}, K, \pi^L,\pi^H) \leq \min_{l \in S_t \backslash k} \mathcal{C}(c_{tl}, K, \pi^L,\pi^H) - \delta.
\end{equation}

\noindent
Therefore, at most one of the $l$ will be active with a non-zero dual cost.
The method can thus be scaled to a greater number of scenarios by adding these
constraints on the fly.\\

With the current discretized distributions containing between 5 and
10 scenarios, the mean and median times to compute the whole solutions
for all candidate capacities
are below $\frac{1}{40}$ of a second.
These metrics are obtained using the BenchmarkTools.jl package\cite{bmtools}.\\

A study of the influence of the $\delta$ parameter is summarized Figure \ref{fig:deltavar}.
For any hour, there always exists a maximum value $\delta_{max}$ above which it
becomes impossible to make a solution better for the lower-level than the
baseline with a difference greater than $\delta_{max}$.

\begin{figure}[ht]
\centering
\includegraphics[width=0.49\textwidth]{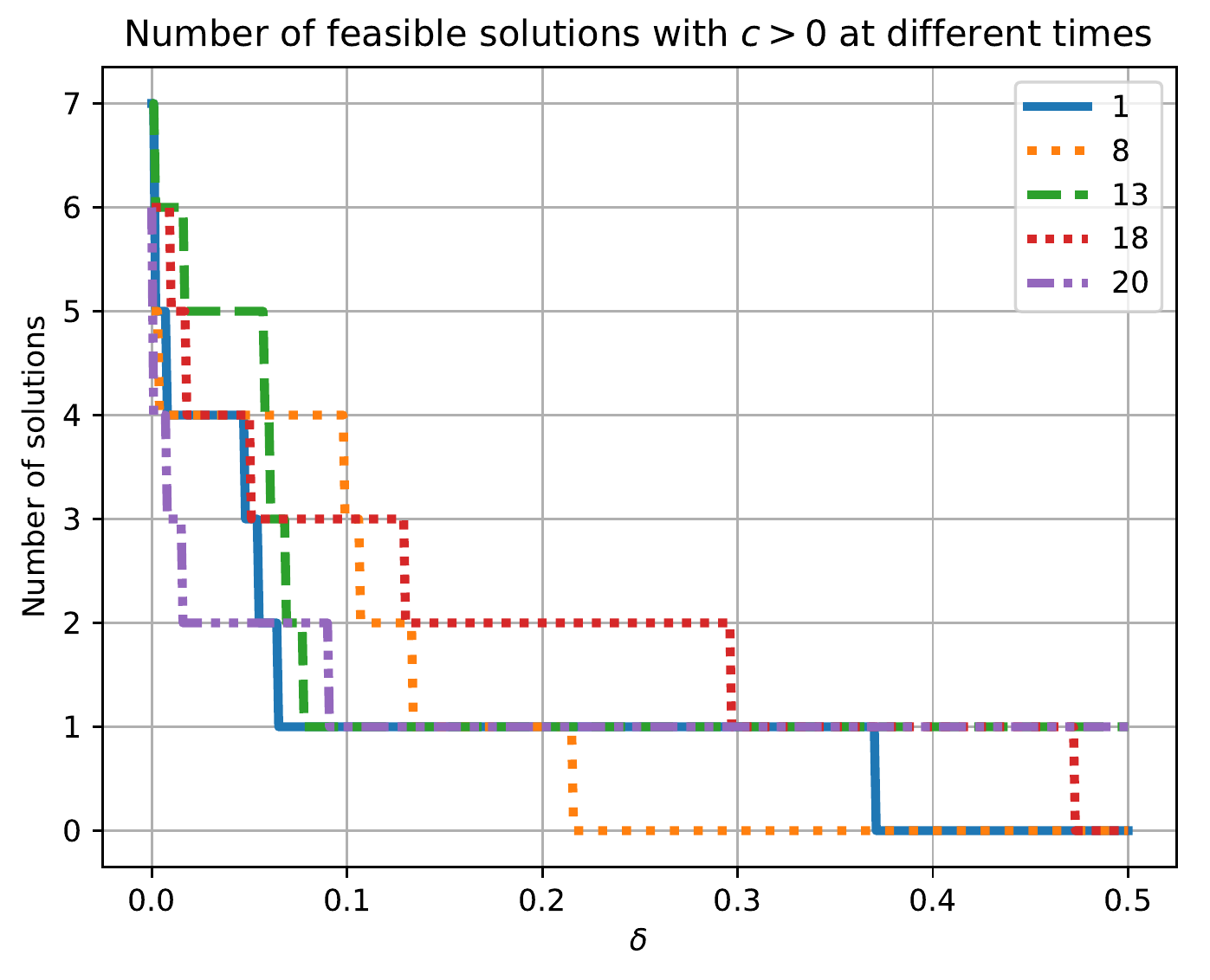}
\caption{Number of non-zero capacity solutions}
\label{fig:deltavar}
\end{figure}

\section{Conclusion}

TLOU is designed to price energy across time and to reflect varying
costs and requirements from the generation side.
Defining two objectives for the supplier,
we built the set of cost-optimal price settings maximizing the guarantee
in a lexicographic fashion. Computations on distributions built from real data
show the effectiveness of the method, requiring a low runtime to compute
the set of solutions. Future research will consider continuous probability
distributions of the consumption and the price-setting problem with multiple
users.

\section{Acknowledgment}

This work was supported by the NSERC Energy Storage Technology (NEST) Network.

\bibliography{refs}{}
\bibliographystyle{nature}
\end{document}